\documentclass{article}
\usepackage{amsmath,amsthm,amssymb}
\usepackage{bbm,url}

\title{Diffusion constants and martingales for senile random walks}
\author{Wouter Kager\footnote{Address: EURANDOM, P.O.~Box 513,
5600~MB Eindhoven, The Netherlands.}}
\date{November 19, 2007}

\newcommand{\N}{\mathbbm{N}}					
\newcommand{\Z}{\mathbbm{Z}}					
\newcommand{\I}{\mathbbm{1}}					
\newcommand{\Prob}{\mathbbm{P}}					
\newcommand{\Exp}{\mathbbm{E}}					
\newcommand{\fld}[1]{\mathcal{#1}}				
\newcommand{\asto}{\xrightarrow{\text{a.s.}}}	
\newcommand{\weak}{\Rightarrow}					

\numberwithin{equation}{section}

\newtheorem{theorem}{Theorem}[section]
\newtheorem{proposition}[theorem]{Proposition}

\begin{document}

\maketitle

\begin{abstract}
	We derive diffusion constants and martingales for senile random walks with 
	the help of a time-change. We provide direct computations of the diffusion 
	constants for the time-changed walks. Alternatively, the values of these 
	constants can be derived from martingales associated with the time-changed 
	walks. Using an inverse time-change, the diffusion constants for senile 
	random walks are then obtained via these martingales. When the walks are 
	diffusive, weak convergence to Brownian motion can be shown using a 
	martingale functional limit theorem.
\end{abstract}

\section{Introduction and general framework}

In this paper we study random walks on $\Z^d$ for dimensions $d\geq1$, which 
can be viewed as time-changes of random walks that were named \emph{senile 
reinforced} and \emph{senile persistent} random walks in~\cite{HS07}. We will 
use this terminology also in this paper, although senile persistent random 
walks were originally introduced and studied under the name of 
\emph{directionally reinforced} random walks in~\cite{HS98, MMW96}. The 
reinforcement of senile random walks is of a different kind than that of more 
traditional edge reinforced random walks, as introduced by Coppersmith and 
Diaconis~\cite{CD86}. For more details and discussion, we refer to the 
introductions and references in~\cite{HS07,HS98,MMW96}, and to the recent 
survey paper~\cite{P07} on reinforced random processes.

Recurrence and transience properties of senile random walks were studied in 
the papers~\cite{HS07,MMW96}, and scaling limits are identified in~\cite{H07, 
HS98}. In this paper, rather than taking the senile random walks themselves as 
our starting point, we start by studying other random walks that are later 
interpreted as time-changes of senile random walks. The idea of looking at 
these time-changed walks has also been used in the mentioned references.  
However, this paper presents a different approach to identifying the diffusion 
constants and weak limits of the walks under study, using mainly martingale 
techniques.

Indeed, below we will provide new, direct calculations of the diffusion 
constants for the time-changed random walks, and we show that these random 
walks are close to martingales (for the persistent case, this has also been 
observed in~\cite{HS98}). Using martingale theory, we can then derive the 
diffusion constants for the senile random walks by an inverse time-change.  
This confirms that Theorem~2.5 in~\cite{HS07} holds under a slightly weaker 
moment condition, as conjectured by the authors. Finally, we will show that 
under appropriate conditions for which the walks are diffusive, weak 
convergence of senile random walks to Brownian motion follows from a 
martingale central limit theorem.

We will now introduce a general framework for the time-changed walks we want 
to study below. Generally, the walks are described by a sequence $W = 
(W_1,W_2,\dots)$ of random variables taking values in~$\Z^d$. For each $n\in\N 
:= \{ 1,2,\dots \}$, we will write $W_n$ (the position at time~$n$) as the sum 
of~$n$ random steps, where the $m$th step ($m\in\N$) has a direction~$D_m$ 
taking values in $\{ e_1,e_2,\dots,e_{2d} \}$, the unit vectors of~$\Z^d$, and 
a length $L_m \in \{ 0,1,2,\dots \}$.

Actually, for the single purpose of relating our walks to senile random walks 
later on, we will write each step length~$L_m$ as a function of a random 
variable~$T_m$ taking values in $\N = \{ 1,2,\dots \}$. These variables~$T_m$ 
are i.i.d.\ (hence, so are the step lengths) and define the random time-change 
linking our random walks to senile random walks. Below, we will use the 
notation~$T$ for a generic variable distributed as any one of the~$T_m$. The 
distribution of the random times~$T_m$ is specified in terms of a function $f: 
\N \to [-1,\infty)$ (the reinforcement function) by $\Prob(T \geq 1) = 1$ and
\begin{align}
\label{T}
	\Prob(T \geq k) = \prod_{l=1}^{k-1} \frac{1+f(l)}{2d+f(l)}
		\qquad \text{for } k = 2,3,\dots
\end{align}
This specific form of the distribution of the~$T_m$ is introduced only to make 
the link with senile random walk. For now, we do not put any restrictions on 
the function~$f$, but later on, we will require that either $\Exp(T)$ is 
finite or both $\Exp\bigl( T^2 \bigr)$ and $\Exp(T)$ are finite, depending on 
whether we consider the reinforced or the persistent case.

Thus, following the description above, we can write
\begin{equation}
\label{rwW}
	W_n := \sum_{m=1}^n D_m \, L_m \qquad \text{for all }n\in\N,
\end{equation}
where the laws of the $D_m$ and~$L_m$ are yet to be specified. In sections 
\ref{sec:persistent} and~\ref{sec:reinforced} we consider two specific 
instances of this general class of random walks, related to senile persistent 
and senile reinforced random walks, respectively. Our first aim will be to 
compute the \emph{diffusion constants} for these walks, which for a general 
walk $X = (X_1,X_2,\dots)$ is defined by
\begin{equation}
\label{dcX}
	C_X := \lim_{n\to\infty} \frac{1}{n} \, \Exp\bigl( |X_n|^2 \bigr),
\end{equation}
provided the limit exists and is finite. To find the diffusion constants for 
the senile random walks, we will then make use of martingales associated with 
the time-changed walks, and these martingales will also be used to prove weak 
convergence to Brownian motion when the senile walks are diffusive.

\section{The persistent case}
\label{sec:persistent}

We start with the persistent case, for which the definition of the walk is 
somewhat easier than in the reinforced case, but the analysis is harder. In 
this case, we take
\begin{equation}
	\label{Lp}
	L_m = T_m \qquad \text{for all } m\in\N,
\end{equation}
and the directions of different steps obey the rule that the direction at each 
step has to  be different from the direction at the previous step, but all 
remaining choices of direction are equally likely. Formally, this means that 
the directions~$D_m$ satisfy
\begin{equation}
	\label{D1p}
	\Prob(D_1 = e_i) = \frac{1}{2d} \qquad \text{for each } i=1,2,\ldots,2d,
\end{equation}
and for all $m\in\N$,
\begin{equation}
	\label{Dp}
	\Prob(D_{m+1} = e_i \mid D_m) \\
	 = \frac{1}{2d-1} \, \I(D_m\neq e_i) \quad \text{for each } 
	 i=1,2,\ldots,2d,
\end{equation}
where $\I(A)$ is the indicator of the event~$A$. Equations 
\eqref{Lp}--\eqref{Dp} completely specify the law of the random walk defined 
by~\eqref{rwW}. For the remainder of this section we will write $W^p = (W^p_1, 
W^p_2, \dots)$ for this walk, where the superscript~$p$ is used to single out 
the \emph{persistent} case studied here.

\subsection{Direct calculation of the diffusion constant}

We will now provide a direct calculation of the diffusion constant for the 
random walk~$W^p$ defined above. It will be clear from the computation that we 
have to require that $\Exp\bigl( T^2 \bigr) < \infty$ (which implies $\Exp(T) 
< \infty$). The diffusion constant is then given by the following 
proposition.

\begin{proposition}
\label{pro:dcWp}
	Assume $\Exp\bigl( T^2 \bigr) < \infty$. Then the diffusion constant of 
	the random walk~$W^p$ is given by
	\[
		C^p := \lim_{n\to\infty} \frac{1}{n} \, \Exp\bigl( |W^p_n|^2 \bigr) = 
		\frac{d \, \Exp\bigl( T^2 \bigr)-\Exp(T)^2}{d}.
	\]
\end{proposition}

\begin{proof}
	It is easy to see that
	\begin{equation}
	\label{dcWp}
		\Exp\bigl( |W^p_n|^2 \bigr) = n \, \Exp\bigl( T^2 \bigr) + 2 \, 
		\sum_{k=1}^{n-1} \sum_{m=1}^{n-k} \Exp( D_m \cdot D_{m+k} \, L_m 
		L_{m+k} ),
	\end{equation}
	where for all $m,k \geq 1$, by independence of the step lengths,
	\begin{equation}
		\Exp( D_m \cdot D_{m+k} \, L_m L_{m+k} ) = \Exp(T)^2 \, \Exp( D_m 
		\cdot D_{m+k} ).
	\end{equation}
	Now note that on the event $E_{mk} := \{D_{m+k-1} \cdot D_{m+k} = 0\}$, 
	$D_m \cdot D_{m+k}$ takes on the values~$\pm1$ with equal probabilities, 
	by~\eqref{Dp}. On the other hand, on the complementary event $E_{mk}^c$, 
	we have that $D_{m+k} = -D_{m+k-1}$. Therefore, using independence again,
	\begin{equation}
	\begin{split}
		\Exp( D_m \cdot D_{m+k} )
		&= \Exp\bigl( D_m \cdot D_{m+k} \, \I(E_{mk}) \bigr)
			- \Exp\bigl( D_m \cdot D_{m+k-1} \, \I(E_{mk}^c) \bigr) \\
		&= -\frac{1}{2d-1} \, \Exp( D_m \cdot D_{m+k-1} ).
	\end{split}
	\end{equation}
	Iterating this recursion relation, it follows that
	\begin{equation}
		\Exp\left( D_m \cdot D_{m+k} \right)
		= \left( \frac{-1}{2d-1} \right)^k.
	\end{equation}
	Plugging this expression into~\eqref{dcWp}, we obtain
	\begin{equation}
	\begin{split}
		\Exp\bigl( |W^p_n|^2 \bigr)
		&= n \, \Exp\bigl( T^2 \bigr) + 2 \, \Exp(T)^2 \, \sum_{k=1}^{n-1} \, 
		(n-k) \, \left( \frac{-1}{2d-1} \right)^k \\
		&= n \, \Exp\bigl( T^2 \bigr) - \Exp(T)^2 \, \left( \frac{n}{d} + 
		\frac{2d-1}{2d^2} \, \left[ \left( \frac{-1}{2d-1} \right)^n - 1 
		\right] \right).
	\end{split}
	\end{equation}
	By~\eqref{dcX}, this equation identifies the value of the diffusion 
	constant if we divide by~$n$ and take the limit $n\to\infty$.
\end{proof}

\subsection{Martingales for the persistent random walk}

The purpose of this subsection is to show that the walk~$W^p$ is within 
bounded distance from a martingale at each step. More precisely, we will see 
that adding a correction of constant length to each position~$W^p_n$ gives us 
a martingale. In fact, Proposition~\ref{pro:Mp} below identifies a second 
martingale by direct calculation, which can be used to provide an alternative 
derivation of the diffusion constant for the walk~$W^p$.

To state our result, we introduce the filtration $\{\fld{F}_n : n\in\N\}$, 
where
\begin{equation}
	\fld{F}_n := \sigma(D_1,T_1,D_2,T_2,\dots,D_n,T_n) \qquad \text{for all } 
	n\in\N.
\end{equation}
Now define a new random walk $M^p$ by
\begin{equation}
\label{Mp}
	M^p_n := W^p_n - \frac{\Exp(T)}{2d} \, D_n \qquad \text{for all } n\in\N.
\end{equation}
As before, we assume that $\Exp\bigl( T^2 \bigr) < \infty$. Then the following 
proposition identifies two martingales associated with the walk~$W^p$.

\begin{proposition}
\label{pro:Mp}
	Let $C^p$ be the diffusion constant appearing in 
	Proposition~\ref{pro:dcWp}. Then $\{ (M^p_n, \fld{F}_n) : n\in\N \}$ and 
	$\bigl\{ \bigl( |M^p_n|^2 - n \, C^p, \fld{F}_n \bigr) : n\in\N \bigr\}$ 
	are martingales.
\end{proposition}

\begin{proof}
	The essential ingredients for the proof are: (i) that the events $E_{n1} 
	:= \{D_n \cdot D_{n+1} = 0\}$ and its complement $E_{n1}^c$ are 
	independent of the events in~$\fld{F}_n$, (ii) that on the event 
	$E_{n1}^c$, $D_{n+1} = -D_n$, and (iii) that on the event $E_{n1}$, 
	$D_{n+1}$ is distributed symmetrically (orthogonal to $D_n$). Observing 
	that
	\begin{equation}
	\label{Mpinc}
		M^p_{n+1} = M^p_n + \frac{\Exp(T)}{2d} \, D_n  + D_{n+1} \left( 
		L_{n+1} - \frac{\Exp(T)}{2d}\right),
	\end{equation}
	it is then not difficult to verify that
	\begin{equation}
	\begin{split}
		\Exp\bigl( M^p_{n+1} \bigm| \fld{F}_n \bigr)
		&= \Exp\bigl( M^p_{n+1} \, \I(E_{n1}) \bigm| \fld{F}_n \bigr)
			+ \Exp\bigl( M^p_{n+1} \, \I(E_{n1}^c) \bigm| \fld{F}_n \bigr) \\
		&= M^p_n.
	\end{split}
	\end{equation}

	Next we use~\eqref{Mpinc} again, as well as $|D_n|^2 = 1$, to compute
	\begin{multline}
		|M^p_{n+1}|^2 = |M^p_n|^2 + \frac{\Exp(T)^2}{2d^2} + L_{n+1}^2 - 
		\frac{\Exp(T)}{d} \, L_{n+1}  \\
		\null + D_n \cdot M^p_n \, \frac{\Exp(T)}{d} + 2 \, D_{n+1} \cdot 
		M^p_n \, \left( L_{n+1} - \frac{\Exp(T)}{2d} \right) \\
		\null + \frac{\Exp(T)}{d} \, D_n \cdot D_{n+1} \left( L_{n+1} - 
		\frac{\Exp(T)}{2d} \right).
	\end{multline}
	In the same way as before, a straightforward calculation now leads to
	\begin{equation}
	\begin{split}
		\Exp\bigl( |M^p_{n+1}|^2 \bigm| \fld{F}_n \bigr)
		&= \Exp\bigl( |M^p_{n+1}|^2 \, \I(E_{n1}) \bigm| \fld{F}_n \bigr)
			+ \Exp\bigl( |M^p_{n+1}|^2 \, \I(E_{n1}^c)
				\bigm| \fld{F}_n \bigr) \\
		&= |M^p_n|^2 + \frac{d \, \Exp\bigl( T^2 \bigr) - \Exp(T)^2}{d},
	\end{split}
	\end{equation}
	confirming the proposition.
\end{proof}

\subsection{Connection with senile persistent random walk}

As alluded to in the introduction, the random walk~$W^p$ studied above can be 
seen as a time-change of another random walk $S^p$, called \emph{senile 
persistent random walk}, sampled at the random times
\begin{equation}
	\tau_n := \sum_{k=1}^n T_k \qquad \text{for all } n\in\N.
\end{equation}
The connection between the two walks is best established through the inverse 
of this time-change. That is, we introduce the random map $\tau^{-1} : 
\N\to\N$ by setting
\begin{equation}
	\tau^{-1}_n := \inf\{ m\in\N : \tau_m \geq n \} \qquad \text{for each } 
	n\in\N.
\end{equation}
Thus, for any point~$\omega$ of the sample space, $\tau^{-1}_n(\omega)$ is the 
time~$m$ such that $\tau_{m-1}(\omega)$ is less than~$n$ and $\tau_m(\omega)$ 
is at least~$n$. Note that for every $n\in\N$, $\tau^{-1}_n$ is a stopping 
time with respect to the filtration $\{ \fld{F}_n : n\in\N \}$, since (setting 
$\tau_0 := 0$)
\begin{equation}
	\bigl\{ \tau^{-1}_n \leq k \bigr\} = \bigcup_{m=1}^k \{\tau_{m-1} < n \leq 
	\tau_m\} = \{n \leq \tau_k\} \in \fld{F}_k.
\end{equation}
We also remark that $\tau^{-1}_n \leq n$ a.s., since $\tau_n$ is necessarily 
at least equal to~$n$.

The senile persistent random walk~$S^p$ can now be defined by
\begin{equation}
\label{Sp}
	S^p_n := W^p_{\tau^{-1}_n} + D_{\tau^{-1}_n} \, \bigl( n - 
	\tau_{\tau^{-1}_n} \bigr) \qquad \text{for all } n = 1,2,\dots,
\end{equation}
where $\tau_{\tau^{-1}_n} = \sum_{m=1}^{\tau^{-1}_n} T_m$. It may not be 
obvious from this formal definition how the walk~$S^p$ behaves, so let us 
discuss this in more detail. First observe that $S^p_{\tau_n} = W^p_n$, so 
that we can indeed interpret $W^p$ as the senile random walk~$S^p$ sampled at 
the times $\tau_n$. Next we note that by~\eqref{Sp}, in between times 
$\tau_{n-1}$ and $\tau_n$, the walk moves in a straight line from the position 
$W_{n-1}$ to $W_n$, taking steps of unit length. Therefore, we see that the 
random walk~$S^p$ is a walk which persists to move in a given direction for a 
random time distributed like~$T$, then chooses a new direction uniformly at 
random, moves in that direction for a random time distributed again like~$T$, 
and so on.

It is now instructive to interpret the role of the function~$f$ appearing in 
the distribution~\eqref{T} of the random times~$T_n$ from the behaviour of the 
walk~$S^p$. Looking at equation~\eqref{T}, we see that the walk~$S^p$, after 
having moved in the same direction for~$n$ steps, chooses to make the next 
step again in the same direction with a probability given by $\bigl( 1+f(n) 
\bigr) / \bigl( 2d+f(n) \bigr)$. Furthermore, all other choices of direction 
for the next step are equally likely. This description of the walk~$S^p$ 
corresponds to how the model was originally defined in~\cite{MMW96}.

Our next objective is to find the diffusion constant for the senile persistent 
random walk~$S^p$. It is given by the following result.

\begin{proposition}
\label{pro:dcSp}
	Suppose that $\Exp\bigl( T^2 \bigr) < \infty$. Then the diffusion constant 
	of the senile persistent random walk~$S^p$ is given by
	\[
		\lim_{n\to\infty} \frac{1}{n} \, \Exp\bigl( |S^p_n|^2 \bigr)
		= \frac{d \, \Exp\bigl( T^2 \bigr) - \Exp(T)^2}{d \, \Exp(T)}
		= \frac{C^p}{\Exp(T)}.
	\]
\end{proposition}

\begin{proof}
	The key observation is that at time~$n$, $S^p$ is not far from $M^p$ at 
	the stopping time~$\tau^{-1}_n$. To be precise, from the definitions 
	\eqref{Mp} and~\eqref{Sp} we see that for all $n\in\N$,
	\begin{equation}
	\label{prf:dcSp1}
		S^p_n = M^p_{\tau^{-1}_n} + D_{\tau^{-1}_n} \left( \frac{\Exp(T)}{2d} 
		+ n - \tau_{\tau^{-1}_n} \right) =: M^p_{\tau^{-1}_n} + X^p_n,
	\end{equation}
	where we have introduced $X^p_n$ to denote the difference between $S^p_n$ 
	and $M^p_{\tau^{-1}_n}$. By the triangle inequality and H\"older's 
	inequality, we then have
	\begin{equation}
		\Bigl| \Exp\bigl( |S^p_n|^2 - | M^p_{\tau^{-1}_n} |^2 - |X^p_n|^2 
		\bigr) \Bigr| \leq 2 \, \sqrt{ \Exp\bigl( |\smash[b]{ 
		M^p_{\tau^{-1}_n} }|^2 \bigr) \Exp\bigl( |X^p_n|^2 \bigr) }.
	\end{equation}
	Therefore, to prove Proposition~\ref{pro:dcSp}, it suffices to show that, 
	on the one hand,
	\begin{equation}
	\label{prf:dcSp2}
		\lim_{n\to\infty} \frac{1}{n} \, \Exp\bigl( |M^p_{\tau^{-1}_n}|^2 
		\bigr) = \frac{d \, \Exp\bigl( T^2 \bigr) - \Exp(T)^2}{d \, \Exp(T)},
	\end{equation}
	and on the other hand,
	\begin{equation}
	\label{prf:dcSp3}
		\lim_{n\to\infty} \frac{1}{n} \, \Exp\bigl( T^2_{\tau^{-1}_n} \bigr) = 
		0.
	\end{equation}
	We will first show \eqref{prf:dcSp2} by appealing to 
	Proposition~\ref{pro:Mp} and the law of large numbers for the random 
	times~$\tau_n$, before we prove~\eqref{prf:dcSp3}.

	To ease the notation, we shall write~$C^p$ for the diffusion constant of 
	the walk~$W^p$ appearing in Proposition~\ref{pro:dcWp}. Now we observe 
	that (by standard martingale theory, see e.g.~\cite[Theorem~6.7.3]{AD00}) 
	Proposition~\ref{pro:Mp} and the fact that $\tau^{-1}_1, \tau^{-1}_2, 
	\dots$ is an increasing sequence of stopping times imply that the process
	\begin{equation}
		\bigl\{ |M^p_{\tau^{-1}_n}|^2 - {\tau^{-1}_n} \, C^p : n\in\N \bigr\}
	\end{equation}
	is a martingale with respect to the filtration $\bigl\{ 
	\fld{F}_{\tau^{-1}_n} : n\in\N \bigr\}$ defined by
	\begin{equation}
	\label{Ftauinvn}
		\fld{F}_{\tau^{-1}_n} := \bigl\{ A\in\fld{F} : A \cap \bigl\{ 
		\tau^{-1}_n \leq k \bigr\} \in \fld{F}_k \text{ for all } 
		k=1,2,\dots,n \bigr\}.
	\end{equation}
	Therefore,
	\begin{equation}
	\label{prf:dcSp4}
		\lim_{n\to\infty} \frac{1}{n} \, \Exp\bigl( |M^p_{\tau^{-1}_n}|^2 - 
		{\tau^{-1}_n} \, C^p \bigr) = \lim_{n\to\infty} \frac{1}{n} \, 
		\Exp\bigl( |M^p_1|^2 - C^p \bigr) = 0.
	\end{equation}
	But the strong law of large numbers dictates that $n^{-1} \tau_n \asto 
	\Exp(T)$, from which it follows that $n^{-1} \tau^{-1}_n \asto 
	\Exp(T)^{-1}$. Since $\tau^{-1}_n \leq n$ a.s., we therefore have
	\begin{equation}
		\lim_{n\to\infty} \frac{1}{n} \, \Exp\bigl( \tau^{-1}_n \bigr) = 
		\frac{1}{\Exp(T)}
	\end{equation}
	by bounded convergence. Together with~\eqref{prf:dcSp4}, this 
	implies~\eqref{prf:dcSp2}.

	It remains to show~\eqref{prf:dcSp3}. To this end, observe that 
	by~\eqref{prf:dcSp1}, $|X^p_n|$ is bounded by the sum of a constant and 
	the term $\tau_{\tau^{-1}_n} - n$, which takes values between $0$ and
	$T_{\tau^{-1}_n}$. Because $\tau^{-1}_n\leq n$, it therefore suffices to 
	show that $Y_n = \max_{k\leq n} T_k / \sqrt{n}$ converges to~$0$ in 
	probability and that the $Y_n^2$ are uniformly integrable. But by Boole's 
	and Markov's inequalities, for $\epsilon>0$,
	\begin{equation}
		\Prob\Bigl( \max_{k\leq n} T_k > \epsilon \sqrt{n} \bigr)
		\leq n \, \Prob(T > \epsilon \sqrt{n}) \leq \frac{1}{\epsilon^2} \,
		\Exp\bigl( T^2 \, \I(T^2 > \epsilon^2 n) \bigr) \to_n 0,
	\end{equation}
	because we are assuming that $\Exp\bigl(T^2\bigr) < \infty$. This shows 
	that $Y_n \to 0$ in probability. Uniform integrability of the $Y_n^2$ 
	follows from the fact that the $Y_n^2$ are bounded by $\tfrac{1}{n} 
	\sum_{k=1}^n T_k^2$, which are uniformly integrable because they have mean 
	$\Exp\bigl(T^2\bigr)$, are positive, and converge almost surely to 
	$\Exp\bigl(T^2\bigr)$. This implies~\eqref{prf:dcSp3}, and completes the 
	proof of Proposition~\ref{pro:dcSp}.
\end{proof}

\subsection{Weak convergence to Brownian motion}

We will now show weak convergence of~$S^p$ to Brownian motion, by applying a 
martingale functional limit theorem to the martingale $\{ M^p_{\tau^{-1}_n} : 
n\in\N \}$ studied above. We follow Billingsley~\cite[Section~18]{Bill99}.  
Let $D[0,\infty)$ be the metric space of right-continuous real functions 
on~$[0,\infty)$ with left-hand limits which has the Skorohod topology, as 
in~\cite[Section~16]{Bill99}. Generally, we will denote by~$W$ standard 
Brownian motion on any functional space under consideration, and we write 
$\weak_n$ to denote weak convergence with~$n$. Setting $S^p_0 := 0$ for the 
senile persistent random walk, the following holds:

\begin{theorem}
\label{thm:Spweak}
	Assume $\Exp\bigl(T^2\bigr)<\infty$. For every $t\geq0$ and $n\in\N$, 
	define
	\begin{equation}
		Z^n_t := \sqrt{\frac{d \, \Exp(T)}{n \, C^p}} \:
		S^p_{\lfloor nt \rfloor},
	\end{equation}
	where $C^p$ is the diffusion constant of~$W^p$ appearing in 
	Proposition~\ref{pro:dcWp}.	Then $Z^n \weak_n W$ in the sense 
	of~$D[0,\infty)^d$.
\end{theorem}

\begin{proof}
	First we recall that $S^p_n$ is close to $M^p_{\tau^{-1}_n}$, as expressed  
	by~\eqref{prf:dcSp1} in the proof of Proposition~\ref{pro:dcSp}. In fact, 
	the proof of Proposition~\ref{pro:dcSp} shows that $\sup_{0\leq s\leq t} 
	\bigl| X^p_{\lfloor ns \rfloor} \bigr| / \sqrt{n}$ converges to~$0$ in 
	probability for every fixed $t>0$. Therefore, it suffices to prove weak 
	convergence to Brownian motion for $M^p_{\tau^{-1}_n}$ instead 
	of~$S^p_n$.

	We recall that $\{ M^p_{\tau^{-1}_n} : n\in\N \}$ is a martingale with 
	respect to the filtration $\fld{F} := \{ \fld{F}_{\tau^{-1}_n} : n\in\N 
	\}$ defined by~\eqref{Ftauinvn}. Let us now write $M^i_{\tau^{-1}_n}$, $i 
	= 1,2,\dots,d$, for the one-dimensional marginals of $M^p_{\tau^{-1}_n}$.  
	For each $n\in\N$, define
	\begin{align}
		\xi^i_{n1}
		&:= \sqrt{\frac{d \, \Exp(T)}{n \, C^p}} \, M^i_{\tau^{-1}_1}
		&\quad i=1,2,\dots,d; \\
		\xi^i_{nk}
		&:= \sqrt{\frac{d \, \Exp(T)}{n \, C^p}} \,
			\Bigl( M^i_{\tau^{-1}_k} - M^i_{\tau^{-1}_{k-1}} \Bigr)
		&\quad i = 1,2,\dots,d; \quad k = 2,3,\dots
	\end{align}
	Then for each~$i$, the $\xi^i_{nk}$ form a triangular array of martingale 
	differences with respect to the filtration~$\fld{F} := \{ 
	\fld{F}_{\tau^{-1}_k} : k\in\N \}$.

	By~\eqref{Mpinc} we have for $k\geq2$
	\begin{equation}
	\begin{split}
		|\xi^i_{nk}|
		&= \sqrt{\frac{d \, \Exp(T)}{n \, C^p}} \,
			\I(\tau^{-1}_k \neq \tau^{-1}_{k-1}) \, \left|
			D^i_{\tau^{-1}_k} \, L_{\tau^{-1}_k} +
			\frac{\Exp(T)}{2d} \, \left(
				D^i_{\tau^{-1}_{k-1}} - D^i_{\tau^{-1}_k}
			\right) \right| \\
		&\leq \sqrt{\frac{d \, \Exp(T)}{n \, C^p}} \, \sum_{l\leq k}
			\I(\tau^{-1}_k=l, \tau^{-1}_{k-1}=l-1) \left(
			L_l + \frac{\Exp(T)}{d} \right) \\
		&= \sqrt{\frac{d \, \Exp(T)}{n \, C^p}} \, \sum_{l\leq k}
			\I(\tau_{l-1}=k-1) \left( T_l + \frac{\Exp(T)}{d} \right).
	\end{split}
	\end{equation}
	Setting $\tau_0 := 0$, it is clear that this bound also holds for $k=1$.

	Now fix $\epsilon>0$ and set $\delta := \tfrac{1}{2} \epsilon \sqrt{C^p / 
	d \, \Exp(T)}$. Then from the bound on $|\xi^i_{nk}|$, it follows that for 
	$n$ sufficiently large we have that
	\begin{multline}
		\sum_{k\leq nt} \Exp\bigl( (\xi^i_{nk})^2 \,
			\I(|\xi^i_{nk}| \geq \epsilon) \bigr) \\
		\leq \frac{2d \, \Exp(T)}{n \, C^p} \,
			\Exp\biggl[ \,
				\sum_{k\leq nt} \sum_{l\leq k}
				\I(\tau_{l-1}=k-1) \, T_l^2 \, \I(T_l \geq \delta\sqrt{n})
			\biggr].
	\end{multline}
	Interchanging the order of summation and using that $\sum_{k=l}^{\lfloor 
	nt \rfloor} \I(\tau_{l-1}=k-1) \leq 1$, we arrive at
	\begin{equation}
		\sum_{k\leq nt} \Exp\bigl( (\xi^i_{nk})^2 \,
			\I(|\xi^i_{nk}| \geq \epsilon) \bigr)
		\leq \frac{2d \, \Exp(T)}{n \, C^p} \, \sum_{l\leq nt}
			\Exp\left[ T_l^2 \, \I(T_l \geq \delta\sqrt{n}) \right].
	\end{equation}
	Since the $T_l$ are i.i.d.\ and $\Exp(T^2) < \infty$, we conclude that for 
	every $t\geq0$,
	\begin{equation}
	\label{BillCond2}
		\sum_{k\leq nt} \Exp\bigl( (\xi^i_{nk})^2 \,
			\I(|\xi^i_{nk}| \geq \epsilon) \bigr) \to_n 0.
	\end{equation}
	
	Now put
	\begin{equation}
	\label{sigma}
		(\sigma^i_{nk})^2 :=
			\Exp\bigl( (\xi^i_{nk})^2 \bigm| \fld{F}_{\tau^{-1}_{k-1}} \bigr),
	\end{equation}
	where we define $\fld{F}_{\tau^{-1}_0}$ to be the trivial $\sigma$-field 
	$\{ \varnothing,\Omega \}$. By Proposition~\ref{pro:Mp} and the symmetry 
	of our random walks, for every $i = 1,2,\dots,d$,
	\begin{equation}
		\Bigl\{ \Bigl( \bigl( M^i_{\tau^{-1}_n} \bigr)^2 - \tau^{-1}_n C^p / 
		d, \fld{F}_{\tau^{-1}_n} \Bigr): n\in\N \Bigr\}
	\end{equation}
	is a martingale. Therefore, for $k\geq2$,
	\begin{equation}
	\begin{split}
		(\sigma^i_{nk})^2
		&= \frac{d \, \Exp(T)}{n \, C^p} \, \Exp\Bigl( (M^i_{\tau^{-1}_k})^2 + 
		(M^i_{\tau^{-1}_{k-1}})^2 - 2 M^i_{\tau^{-1}_k} M^i_{\tau^{-1}_{k-1}} 
		\Bigm| \fld{F}_{\tau^{-1}_{k-1}} \Bigr) \\
		&= \frac{\Exp(T)}{n} \, \Bigl( \Exp\bigl( \tau^{-1}_k \bigm| 
		\fld{F}_{\tau^{-1}_{k-1}} \bigr) - \tau^{-1}_{k-1} \Bigr).
	\end{split}
	\end{equation}
	Next we observe that for all $l,m\in\N$ (considering $l>m$ and $l\leq m$ 
	in turn),
	\begin{equation}
		\{ \tau^{-1}_k \leq l \} \cap \{ \tau^{-1}_{k-1} \leq m \}
		= \{ \tau_l \geq k \} \cap \{ \tau_m \geq k-1 \} \in \fld{F}_m.
	\end{equation}
	It follows by~\eqref{Ftauinvn} that $\{ \tau^{-1}_k \leq l \} \in 
	\fld{F}_{\tau^{-1}_{k-1}}$ for all~$l$, and hence, that the random 
	variable $\tau^{-1}_k$ is in fact $\fld{F}_{\tau^{-1}_{k-1}}$-measurable.  
	Therefore, for $k\geq2$,
	\begin{equation}
		(\sigma^i_{nk})^2 = \frac{\Exp(T)}{n} \, \bigl( \tau^{-1}_k - 
		\tau^{-1}_{k-1} \bigr).
	\end{equation}
	By the strong law of large numbers, it immediately follows that for every 
	$t\geq0$,
	\begin{equation}
	\label{BillCond1}
		\sum_{k\leq nt} (\sigma^i_{nk})^2 = \sum_{k\leq nt} \Exp\bigl( 
		(\xi^i_{nk})^2 \bigm| \fld{F}_{\tau^{-1}_{k-1}} \bigr) \asto_n t.
	\end{equation}

	Now write (taking $M^i_0 := 0$)
	\begin{equation}
		Y^{ni}_t := \sum_{k\leq nt} \xi^i_{nk}
		= \sqrt{\frac{d \, \Exp(T)}{n \, C^p}} \:
		M^i_{\tau^{-1}_{\lfloor nt \rfloor}}.
	\end{equation}
	Then Theorem~18.2 in~\cite{Bill99} states that because \eqref{BillCond1} 
	and~\eqref{BillCond2} both hold, $Y^{ni} \weak_n W$ in the sense of 
	$D[0,\infty)$. In other words, we have shown that the one-dimensional 
	marginals~$Y^{ni}$ converge weakly to Brownian motion. We now want to 
	extend this to weak convergence of $Y^n = (Y^{n1}, \dots, Y^{nd})$. The 
	proof of Theorem~18.2 in~\cite{Bill99} shows that for each~$i$ the laws of 
	the one-dimensional marginals~$Y^{ni}$ form a tight family. But since the 
	product of compact sets in $D[0,\infty)$ is a compact set, this implies 
	tightness of the family of laws of the $Y^n$. It remains to show that all 
	finite-dimensional distributions of~$Y^n$ converge to those of 
	$d$-dimensional Brownian motion.

	To show this, we need the additional result that for $i\neq j$ and all 
	$n\in\N$,
	\begin{equation}
		\Exp\bigl( (M^i_{n+1} - M^i_n)(M^j_{n+1} - M^j_n) \bigm| \fld{F}_n 
		\bigr) = 0.
	\end{equation}
	This can be seen by using~\eqref{Mpinc} and noting that on the event 
	$\{D_n \cdot D_{n+1} \neq 0\}$, $D_{n+1} = -D_n$ whereas on the event 
	$\{D_n \cdot D_{n+1} = 0\}$, $D^i_{n+1} D^j_n$ takes on each of the values 
	$\pm1$ with equal probabilities. It follows that for $i\neq j$ and 
	fixed~$n$, the $\xi^i_{nk} \xi^j_{nk}$ are martingale differences with 
	respect to the filtration~$\fld{F}$.

	Now fix $s,t\geq0$ and let $(a_1,\ldots,a_d)$ and $(b_1,\ldots,b_d)$ be 
	arbitrary vectors of real numbers. Define $\eta_{nk}$ as $\sum_i (a_i+b_i) 
	\xi^i_{nk}$ for $k \leq \lfloor ns \rfloor$ and as $\sum_i b_i \xi^i_{nk}$ 
	for $\lfloor ns \rfloor < k \leq \lfloor n(s+t) \rfloor$.  Then, 
	by~\eqref{BillCond1} and because the $\xi^i_{nk} \xi^j_{nk}$ are 
	martingale differences when $i\neq j$,
	\begin{equation}
		\sum_{k\leq n(s+t)} \Exp\Bigl( \eta_{nk}^2 \Bigm| 
		\fld{F}_{\tau^{-1}_{k-1}} \Bigr)
		\asto_n \sum_i (a_i+b_i)^2 \, s + \sum_i b_i^2 \, t.
	\end{equation}
	Therefore, by \cite[Theorem~18.1]{Bill99}, $\sum_i \bigl( a_iY^{ni}_s + 
	b_iY^{ni}_{s+t} \bigr) \weak_n \sum_i \bigl( a_iW^i_s + b_iW^i_{s+t} 
	\bigr)$, where the $W^i$ are the one-dimensional marginals of 
	$d$-dimensional Brownian motion. But since the $a_i$ and $b_i$ were 
	arbitrary, by the Cram\'er-Wold device $(Y^{n1}_s, \ldots, Y^{nd}_s, 
	Y^{n1}_{s+t}, \ldots, Y^{nd}_{s+t}) \weak_n (W^1_s, \ldots, W^d_s, 
	W^1_{s+t}, \ldots, W^d_{s+t})$. It is easy to see that this argument can 
	be generalized to show that all finite-dimensional distributions of~$Y^n$ 
	converge to those of $d$-dimensional Brownian motion. This completes the 
	proof.
\end{proof}

\section{The reinforced case}
\label{sec:reinforced}

We now turn our attention to the reinforced case, where we set
\begin{equation}
\label{Lr}
	L_m = \I(T_m \text{ is odd}) \qquad \text{for all } m\in\N.
\end{equation}
Thus, the lengths of the steps of the random walk are i.i.d.\ variables taking 
values in $\{0,1\}$. However, the directions of different steps are not 
independent. Namely, the step following a step of length~$0$ may not have the 
same direction as the previous step, and the step following a step of 
length~$1$ may not be in the opposite direction of the previous step. All 
other choices of direction are equally likely. Formally, the directions~$D_m$ 
satisfy
\begin{equation}
\label{D1r}
	\Prob(D_1 = e_i) = \frac{1}{2d} \qquad \text{for each } i=1,2,\ldots,2d,
\end{equation}
and for all $m\in\N$ and each $i=1,2,\ldots,2d$,
\begin{multline}
\label{Dr}
	\Prob(D_{m+1} = e_i \mid D_m,L_m) \\
	\null = \frac{1}{2d-1} \, \I(D_m\neq e_i, L_m=0)
	+ \frac{1}{2d-1} \, \I(D_m\neq -e_i, L_m=1).
\end{multline}
Equations \eqref{Lr}--\eqref{Dr} completely specify the law of the random walk 
defined by~\eqref{rwW}. For the remainder of this section, we will denote this 
walk by $W^r = (W^r_1,W^r_2,\dots)$, where the superscript~$r$ is used to 
identify the \emph{reinforced} case.

From~\eqref{Lr}, it may not come as a surprise that the quantity
\begin{equation}
\label{p}
	p := \Prob(T \text{ is odd})
\end{equation}
plays an important role in the analysis of the reinforced case. In fact, if 
$d=1$ we see from~\eqref{Dr} that the walk has a trivial behaviour if $p=1$, 
since then it keeps moving in the same direction. Let us therefore take the 
opportunity to exclude this special case from the analysis for the remainder 
of this section, so that we don't have to repeat the condition that $p<1$ if 
$d=1$ all the time. Note, however, that the case $p=1$ is perfectly fine and 
nontrivial in higher dimensions. Also, when we consider the time-changed 
walk~$W^r$ there will be no problem in allowing $\Prob(T = \infty) > 0$, where 
we may assume either that ``$T_n$ is odd'' is false, or that ``$T_n$ is odd'' 
is true if $T_n = \infty$, whichever one prefers.

\subsection{Direct calculation of the diffusion constant}

We will now compute the diffusion constant for the random walk~$W^r$ defined 
above. In terms of the parameter $p = \Prob(T \text{ is odd})$, the diffusion 
constant is identified by the following proposition.

\begin{proposition}
\label{pro:dcWr}
	The diffusion constant of the random walk~$W^r$ is given by
	\[
		C^r := \lim_{n\to\infty} \frac{1}{n} \, \Exp\bigl( |W^r_n|^2 \bigr) = 
		\frac{d\,p}{d-p}.
	\]
\end{proposition}

\begin{proof}
	We start from the observation that
	\begin{equation}
	\label{dcWr}
		\Exp\bigl( |W_n^r|^2 \bigr) = n \, p + 2 \, \sum_{k=1}^{n-1} 
		\sum_{m=1}^{n-k} \Exp( D_m\cdot D_{m+k} \, L_m L_{m+k} ),
	\end{equation}
	where, since $D_m \cdot D_{m+k} \, L_m$ is independent of $L_{m+k}$,
	\begin{equation}
	\label{dcWr_a}
		\Exp( D_m\cdot D_{m+k} \, L_m L_{m+k} )
		= p \, \Exp( D_m\cdot D_{m+k} \, L_m ).
	\end{equation}
	Now note that on the event $E_{mk} := \{D_{m+k-1} \cdot D_{m+k} = 0\}$, 
	$D_m \cdot D_{m+k}$ takes on the values $\pm1$ with equal probabilities 
	by~\eqref{Dr}. On the other hand, on the complementary event $E_{mk}^c$ we 
	have that $D_{m+k} = D_{m+k-1} \, (2L_{m+k-1}-1)$. Therefore, again using 
	independence,
	\begin{equation}
	\begin{split}
		\Exp( D_m\cdot D_{m+k} \, L_m )
		&= \Exp\bigl( D_m \cdot D_{m+k-1} \, L_m (2L_{m+k-1}-1) \,
				\I(E_{mk}^c) \bigr) \\
		&= \frac{2p-1}{2d-1} \, \Exp( D_m\cdot D_{m+k-1} \, L_m ).
	\end{split}
	\end{equation}
	Iterating this recursion and using~\eqref{dcWr_a}, we obtain
	\begin{equation}
	\begin{split}
		\Exp(D_m\cdot D_{m+k} \, L_m L_{m+k})
		&= p \, \left( \frac{2p-1}{2d-1} \right)^{k-1} \,
			\Exp(D_m \cdot D_{m+1} \, L_m) \\
		&= \frac{p^2}{2d-1} \, \left( \frac{2p-1}{2d-1} \right)^{k-1}.
	\end{split}
	\end{equation}
	Substituting this result into~\eqref{dcWr} yields
	\begin{equation}
	\begin{split}
		\Exp\bigl( |W_n^r|^2 \bigr)
			&= n \, p + \frac{2p^2}{2d-1} \, \sum_{k=1}^{n-1} \, (n-k) \, 
			\left( \frac{2p-1}{2d-1} \right)^{k-1} \\
			&= n \, \frac{d \, p}{d-p} + \frac{p^2 \, (2d-1)}{2 \, (d-p)^2} \, 
			\left[ \left( \frac{2p-1}{2d-1} \right)^n - 1 \right].
	\end{split}
	\end{equation}
	The value of the diffusion constant for the random walk~$W^r$ follows.
\end{proof}

\subsection{Martingales for the reinforced random walk}

The purpose of this subsection is to identify martingales associated with the 
random walk~$W^r$ introduced above. Our main observation is that if we add a 
correction of constant length (but random direction) to the positions $W^r_n$, 
then we obtain a martingale. To be precise, define
\begin{equation}
\label{Mr}
	M^r_n := W^r_n + \frac{p}{2 \, (d-p)} \, D_n \, (2L_n-1) \qquad \text{for 
	all } n\in\N,
\end{equation}
and let $\{ \fld{F}_n : n\in\N\}$ be the filtration defined by
\begin{equation}
	\fld{F}_n := \sigma(D_1,T_1,D_2,T_2,\dots,D_n,T_n) \qquad \text{for all } 
	n\in\N.
\end{equation}
The following proposition identifies two martingales associated with~$W^r$.

\begin{proposition}
\label{pro:Mr}
	Let $C^r$ be the diffusion constant appearing in 
	Proposition~\ref{pro:dcWr}. Then $\{ (M^r_n, \fld{F}_n) : n\in\N \}$ and 
	$\bigl\{ \bigl( |M^r_n|^2 - n \, C^r, \fld{F}_n \bigr) :  n\in\N	
	\bigr\}$ are martingales.
\end{proposition}

\begin{proof}
	The key observation is that even though the direction $D_{n+1}$ itself 
	depends on $D_n$ and $L_n$, the events $E_{n1} := \{D_n \cdot D_{n+1} = 
	0\}$ and its complement $E_{n1}^c$ are \emph{independent} of the events in 
	$\fld{F}_n$, and have the probabilities $(2d-2)/(2d-1)$ and $1/(2d-1)$, 
	respectively. Moreover, on the event $E_{n1}$, $D_{n+1}$ is distributed 
	symmetrically (orthogonal to $D_n$), and on the event $E_{n1}^c$, we have 
	$D_{n+1} = D_n \, (2L_n - 1)$. Now observe that for all $n\in\N$,
	\begin{multline}
	\label{Mrinc}
		M^r_{n+1} = M^r_n - \frac{p}{2 \, (d-p)} \, D_n \, (2L_n-1) \\
		+ \frac{p}{2 \, (d-p)} \, D_{n+1} \, (2L_{n+1}-1) + D_{n+1} \, 
		L_{n+1}.
	\end{multline}
	A simple computation then yields
	\begin{equation}
	\begin{split}
		\Exp\bigl( M^r_{n+1} \bigm| \fld{F}_n \bigr)
			&= \Exp\bigl( M^r_{n+1} \, \I(E_{n1}) \bigm| \fld{F}_n \bigr)
				+ \Exp\bigl( M^r_{n+1} \, \I(E_{n1}^c)
					\bigm| \fld{F}_n \bigr) \\
			&= M^r_n.
	\end{split}
	\end{equation}

	Likewise, for all $n\in\N$ we can write
	\begin{multline}
		|M^r_{n+1}|^2 = |M^r_n|^2 + \frac{p^2}{2 \, (d-p)^2} + \frac{d}{d-p} 
		\, \I(L_{n+1} = 1) \\
		\null - D_n \cdot M^r_n \, \frac{p}{d-p} \, (2L_n-1) + D_{n+1} \cdot 
		M^r_n \, \left[ \frac{p}{d-p} \, (2L_{n+1}-1) + 2L_{n+1} \right] \\
		\null - \frac{p}{2 \, (d-p)} \, D_n \cdot D_{n+1} \, (2L_n-1) \, 
		\left[ \frac{p}{d-p} \, (2L_{n+1}-1) + 2L_{n+1} \right].
	\end{multline}
	A straightforward computation gives
	\begin{equation}
	\begin{split}
		\Exp\bigl( |M^r_{n+1}|^2 \bigm| \fld{F}_n \bigr)
			&= \Exp\bigl( |M^r_{n+1}|^2 \, \I(E_{n1}) \bigm| \fld{F}_n \bigr)
				+ \Exp\bigl( |M^r_{n+1}|^2 \, \I(E_{n1}^c)
					\bigm| \fld{F}_n \bigr) \\
			&= |M^r_n|^2 + \frac{d \, p}{d-p}.
	\end{split}
	\end{equation}
	This confirms that the two processes of the proposition are martingales 
	with respect to the filtration $\{\fld{F}_n : n\in\N\}$.
\end{proof}

\subsection{Connection with senile reinforced random walk}

Like in the persistent case, the walk $W^r$ can be interpreted as a 
\emph{senile reinforced random walk} sampled at the random times
\begin{equation}
	\tau_n := \sum_{k=1}^n T_k \qquad \text{for all } n\in\N.
\end{equation}
As before, we concentrate on the inverse time-change defined by
\begin{equation}
	\tau^{-1}_n := \inf\{ m\in\N : \tau_m \geq n \} \qquad \text{for each } 
	n\in\N.
\end{equation}
We recall that the random times $\tau^{-1}_n$ are stopping times with  respect 
to the filtration $\{\fld{F}_n : n\in\N\}$, and that $\tau^{-1}_n \leq n$ 
almost surely.

We may define the senile reinforced random walk~$S^r$ on $\Z^d$ by
\begin{equation}
\label{Sr}
	S^r_n := W^r_{\tau^{-1}_n} - D_{\tau^{-1}_n} \, \I\bigl( 
	\tau_{\tau^{-1}_n} - n \text{ is odd} \bigr) \qquad \text{for all } 
	n\in\N,
\end{equation}
where $\tau_{\tau^{-1}_n} = \sum_{m=1}^{\tau^{-1}_n} T_m$. Observe that indeed 
$S^r_{\tau_n} = W^r_n$ for all $n\in\N$, so that we may interpret $W^r$ as the 
senile random walk~$S^r$ sampled at the times~$\tau_n$. Furthermore, by the 
definition~\eqref{Sr}, in between times $\tau^{-1}_{n-1}$ and $\tau^{-1}_n$, 
the walk $S^r$ jumps back and forth between the positions $W^r_{n-1}$ and 
$W^r_n$. Thus, the senile reinforced random walk~$S^r$ is a walk that 
traverses an edge back and forth for a random time distributed like~$T$, then 
selects a new edge uniformly at random and traverses that edge for a random 
time again distributed like~$T$, and so on.

As in the persistent case, this description gives us an interpretation of the 
reinforcement function~$f$ defining the distribution of the random times~$T_n$ 
in~\eqref{T}. Namely, the walk $S^r$ has the property that after it has been 
traversing the same edge back and forth for the last~$n$ steps, it will choose 
to traverse that edge again in the next step with probability $(1+f(n)) / 
(2d+f(n))$. Furthermore, all other choices for the next edge are equally 
likely. This corresponds to the original definition of the model 
in~\cite{HS07}.

At this stage, we should note that the walk gets stuck on an edge in case $T_n 
= \infty$ for some $n\in\N$. For the remainder of this section, we will rule 
out this possibility by assuming $\Prob(T = \infty) = 0$. However, we note 
that the following proposition and proof hold perfectly well if $\Exp(T) = 
\infty$, when we take division by $\infty$ to yield~$0$.

\begin{proposition}
\label{pro:dcSr}
	For the senile reinforced random walk~$S^r$, the diffusion constant is 
	given by
	\[
		\lim_{n\to\infty} \frac{1}{n} \, \Exp\bigl( |S^r_n|^2 \bigr)
		= \frac{1}{\Exp(T)} \, \frac{d \, p}{d-p} = \frac{C^r}{\Exp(T)}.
	\]
\end{proposition}

\begin{proof}
	The proof proceeds in the same way as the proof of 
	Proposition~\ref{pro:dcSp}, and is in fact somewhat simpler. First we want 
	to express $S^r$ in terms of the martingale $M^r$. By the definitions 
	\eqref{Mr} and~\eqref{Sr} we have for all $n\in\N$,
	\begin{equation}
	\begin{split}
		S^r_n
		&= M^r_{\tau^{-1}_n} - D_{\tau^{-1}_n} \, \Bigl( \I\bigl( 
		\tau_{\tau^{-1}_n} - n \text{ is odd} \bigr) + \frac{p}{2 \, (d-p)} \, 
		\bigl( 2L_{\tau^{-1}_n} - 1 \bigr) \Bigr) \\
		&=: M^r_{\tau^{-1}_n} + X^r_n,
	\end{split}
	\end{equation}
	where we have introduced $X^r_n$ to denote the difference between $S^r$ at 
	time~$n$ and $M^r$ at time $\tau^{-1}_n$. By the triangle inequality and 
	H\"older's inequality,
	\begin{equation}
		\Bigl| \Exp\bigl( |S^r_n|^2 - |M^r_{\tau^{-1}_n}|^2 - |X^r_n|^2 \bigr) 
		\Bigr| \leq 2 \, \sqrt{ \Exp\bigl( |\smash[b]{ M^r_{\tau^{-1}_n} }|^2 
		\bigr) \Exp\bigl( |X^r_n|^2 \bigr) }.
	\end{equation}
	Since $|X^r_n|$ is bounded by a constant, it therefore suffices to show 
	that
	\begin{equation}
	\label{prf:dcSr1}
		\lim_{n\to\infty} \frac{1}{n} \, \Exp\bigl( |M^r_{\tau^{-1}_n}|^2 
		\bigr) = \frac{C^r}{\Exp(T)}
	\end{equation}
	in order to prove Proposition~\ref{pro:dcSr}.

	But Proposition~\ref{pro:Mr} and the fact that the $\tau^{-1}_n$ are 
	stopping times with respect to the filtration $\{ \fld{F}_n : n\in\N \}$ 
	imply that $\bigl\{ |M^r_{\tau^{-1}_n}|^2 - \tau^{-1}_n \, C^r : n\in\N 
	\bigr\}$ is a martingale with respect to the filtration $\bigl\{ 
	\fld{F}_{\tau^{-1}_n} : n\in\N \bigr\}$ defined by~\eqref{Ftauinvn}.  
	Therefore,
	\begin{equation}
		\lim_{n\to\infty} \frac{1}{n} \, \Exp\Bigl( |M^r_{\tau^{-1}_n}|^2 - 
		\tau^{-1}_n \, C^r \Bigr) = \lim_{n\to\infty} \frac{1}{n} \, 
		\Exp\Bigl( |M^r_1|^2 - C^r \Bigr) = 0.
	\end{equation}
	The strong law of large numbers dictates that $n^{-1}\tau^{-1}_n \asto 
	\Exp(T)^{-1}$, from which we get by bounded convergence that $n^{-1} 
	\Exp\bigl( \tau^{-1}_n \bigr) \to \Exp(T)^{-1}$. Together with the 
	previous result this implies~\eqref{prf:dcSr1}, proving the proposition.
\end{proof}

\subsection{Weak convergence to Brownian motion}

Weak convergence to Brownian motion for the senile reinforced random walk can 
be shown in the same way as for the persistent case, studied in 
Theorem~\ref{thm:Spweak}. As before, we let $D[0,\infty)$ be the metric space 
of right-continuous real functions on $[0,\infty)$ with left-hand limits, and 
we write~$W$ for Brownian motion and $\weak_n$ for weak convergence with~$n$.    
We assume $\Exp(T)<\infty$ to work in the diffusive regime of 
Proposition~\ref{pro:dcSr}. Then, setting $S^r_0 := 0$, the following holds:

\begin{theorem}
\label{thm:Srweak}
	Assume $\Exp(T) < \infty$. For every $t\geq0$ and $n\in\N$, define
	\begin{equation}
		Z^n_t := \sqrt{ \frac{d \, \Exp(T)}{n \, C^r} } \: S^r_{\lfloor nt 
		\rfloor},
	\end{equation}
	where $C^r$ is the diffusion constant of~$W^r$ appearing in 
	Proposition~\ref{pro:dcWr}. Then $Z^n \weak_n W$ in the sense of 
	$D[0,\infty)^d$.
\end{theorem}

\begin{proof}
	For the senile reinforced random walk, the difference $X^r_n$ between 
	$S^r_n$ and $M^r_{\tau^{-1}_n}$, as defined by~\eqref{prf:dcSr1}, is 
	uniformly bounded by a constant. From this it follows that it is 
	sufficient to prove weak convergence to Brownian motion for 
	$M^r_{\tau^{-1}_n}$ instead of $S^r_n$.

	As in the proof of Theorem~\ref{thm:Spweak}, let us write 
	$M^i_{\tau^{-1}_n}$, $i = 1,2,\dots,d$, for the one-dimensional marginals 
	of $M^r_{\tau^{-1}_n}$. For each $n\in\N$, define
	\begin{align}
		\xi^i_{n1}
		&:= \sqrt{\frac{d \, \Exp(T)}{n \, C^r}} \, M^i_{\tau^{-1}_1}
		&\quad i=1,2,\dots,d; \\
		\xi^i_{nk}
		&:= \sqrt{\frac{d \, \Exp(T)}{n \, C^r}} \,
			\Bigl( M^i_{\tau^{-1}_k} - M^i_{\tau^{-1}_{k-1}} \Bigr)
		&\quad i = 1,2,\dots,d; \quad k = 2,3,\dots
	\end{align}
	Then for each~$i$, the $\xi^i_{nk}$ form a triangular array of martingale 
	differences with respect to the filtration $\fld{F} := \{ 
	\fld{F}_{\tau^{-1}_k} : k\in\N \}$.

	By~\eqref{Mrinc} it is clear that all random variables $\sqrt{n} 
	|\xi^i_{nk}|$ are uniformly bounded by a constant. We conclude that for 
	every $\epsilon>0$ and $t\geq0$,
	\begin{equation}
		\sum_{k\leq nt} \Exp\bigl( (\xi^i_{nk})^2 \,
			\I(|\xi^i_{nk}| \geq \epsilon) \bigr) \to_n 0.
	\end{equation}
	Moreover, we can follow the steps \eqref{sigma}--\eqref{BillCond1} in the 
	proof of Theorem~\ref{thm:Spweak} to see that for every $t\geq0$,
	\begin{equation}
		\sum_{k\leq nt} \Exp\bigl( (\xi^i_{nk})^2 \bigm| 
		\fld{F}_{\tau^{-1}_{k-1}} \bigr) \asto_n t.
	\end{equation}
	If we now write, setting $M^r_0 := 0$,
	\begin{equation}
		Y^{ni}_t := \sum_{k\leq nt} \xi^i_{nk}
		= \sqrt{\frac{d \, \Exp(T)}{n \, C^r}} \:
		M^i_{\tau^{-1}_{\lfloor nt \rfloor}},
	\end{equation}
	then Theorem~18.2 in~\cite{Bill99} states that $Y^{ni} \weak_n W$ in the 
	sense of $D[0,\infty)$. In particular, for each~$i$ the laws of the 
	one-dimensional marginals~$Y^{ni}$ form a tight family, which implies that 
	the family of laws of $Y^n = (Y^{n1}, \dots, Y^{nd})$ is tight as well. 
	It remains to show that all finite-dimensional distributions of~$Y^n$ 
	converge to those of $d$-dimensional Brownian motion.

	As in the persistent case, the result will follow if we can show that for 
	$i\neq j$ and all $n\in\N$,
	\begin{equation}
		\Exp\bigl( (M^i_{n+1} - M^i_n)(M^j_{n+1} - M^j_n) \bigm| \fld{F}_n 
		\bigr) = 0.
	\end{equation}
	This can be shown by using~\eqref{Mrinc} and noting that on the event 
	$\{D_n \cdot D_{n+1} \neq 0\}$, $D_{n+1} = D_n \, (2L_n-1)$ whereas on the 
	event $\{D_n \cdot D_{n+1} = 0\}$, $D^i_{n+1} D^j_n$ takes the values 
	$\pm1$ with equal probabilities. It follows that for $i\neq j$ and 
	fixed~$n$, the $\xi^i_{nk} \xi^j_{nk}$ are martingale differences with 
	respect to the filtration~$\fld{F}$. The proof can now be completed as in 
	the persistent case.
\end{proof}

\section*{Acknowledgements}

I would like to thank Mark Holmes for his encouragements to write down these 
notes, Akira Sakai for some useful comments on the manuscript, and both of 
them for acquainting me with senile random walks.


\begin{thebibliography}{1}

\bibitem{AD00}
R.B.~Ash and C.A.~Dol\'eans-Dade.
\newblock {\em Probability \& Measure Theory}.
\newblock Harcourt/Academic Press, Burlington MA, 2nd edition, 1999.

\bibitem{Bill99}
P.~Billingsley.
\newblock {\em Convergence of probability measures}.
\newblock Wiley, New York NY, 2nd edition, 1999.

\bibitem{CD86}
D.~Coppersmith and P.~Diaconis.
\newblock Random walk with reinforcement.
\newblock Unpublished manuscript, 1986.

\bibitem{HS07}
M.~Holmes and A.~Sakai.
\newblock Senile reinforced random walks.
\newblock {\em Stoch.\ Proc.\ Appl.}, {\bf 117}:1519--1539, 2007.
\newblock Preprint available at \url{http://www.stat.auckland.ac.nz/~mholmes}

\bibitem{H07}
M.~Holmes.
\newblock The scaling limit of senile reinforced random walk.
\newblock 2007.
\newblock Preprint available at \url{http://www.stat.auckland.ac.nz/~mholmes}

\bibitem{HS98}
L.~Horv\'ath and Q-M.~Shao.
\newblock Limit Distributions of Directionally Reinforced Random Walks.
\newblock {\em Adv.\ Math.}, {\bf 134}:367--383, 1998.

\bibitem{MMW96}
R.D.~Mauldin, M.~Monticino and H.~von Weizs\"acker.
\newblock Directionally Reinforced Random Walks.
\newblock {\em Adv.\ Math.}, {\bf 117}:239--252, 1996.

\bibitem{P07}
R.~Pemantle.
\newblock A survey of random processes with reinforcement.
\newblock {\em Probab.\ Surv.}, {\bf 4}:1--79, 2007.

\end{thebibliography}
\end{document}